\documentclass{amsart}
\usepackage{amsfonts}
\usepackage{amsmath,amssymb}
\usepackage{amsthm}
\usepackage{amscd}
\usepackage{graphics}
\usepackage{graphicx}
\theoremstyle{definition}{
\newtheorem{Def}{{\rm Definition}}
\newtheorem{Ex}{{\rm Example}}
\newtheorem{Rem}{{\rm Remark}}

}
\theoremstyle{plain}
{
\newtheorem{Cor}{Corollary}
\newtheorem{Prop}{Proposition}
\newtheorem{Thm}{Theorem}
\newtheorem{MainThm}{Main Theorem}

}

\begin{document}
\title[Special generic maps into ${\mathbb{R}}^5$ and studies on the cohomology]{Special generic maps into ${\mathbb{R}}^5$ on closed and simply-connected manifolds and information on the cohomology of the manifolds}
\author{Naoki Kitazawa}
\address{Institute of Mathematics for Industry, Kyushu University, 744 Motooka, Nishi-ku Fukuoka 819-0395, Japan\\
 TEL (Office): +81-92-802-4402 \\
 FAX (Office): +81-92-802-4405 \\
}
\email{n-kitazawa@imi.kyushu-u.ac.jp}
\urladdr{https://naokikitazawa.github.io/NaokiKitazawa.html}

\begin{abstract}
	Morse functions with exactly two singular points on spheres and canonical projections of spheres belong to the class of a certain good class of smooth maps: {\it special generic} maps.
	We mainly investigate information on cohomology of closed and simply-connected manifolds admitting such maps into the $5$-dimensional Euclidean spaces by investigating the embedded curves and submanifolds and their preimages. 

Studies on homology groups for ones into the Euclidean spaces (whose dimensions are lower than $5$ in most cases) have been pioneered by Saeki and Sakuma since 1990s and later by Nishioka and Wrazidlo since 2010s. Recently the author has started pioneering studies on the cohomology for cases where the dimensions of the Euclidean spaces may not be lower than $5$. Our new cases are difficult due to the situation that the dimensions of manifolds we consider are higher. Previously, we have found several restrictions on the cohomology rings. We present new restrictions by new investigations. 

\end{abstract}


\maketitle
\section{Introduction.}
\label{sec:1}

The class of special generic maps is a class of smooth maps whose codimensions are not positive containing the following functions and maps as simplest natural ones.

\begin{itemize}
	\item Morse functions on spheres with exactly two singular points.  
	\item Canonical projections of unit spheres.
\end{itemize}
  
We define {\it special generic} maps. 
A {\it singular} point of a smooth map between two smooth maps is a point in the manifold of the domain where the rank of the differential is smaller than both the dimensions of the two manifolds. The {\it singular} set of the map is the set of all singular points of the map.
  
  \begin{Def}
  	\label{def:1}
  	Let $m \geq n \geq 1$ be integers.
  	A smooth map from an $m$-dimensional manifold with no boundary into an $n$-dimensional one is {\it special generic} if around each singular point there exist suitable local coordinates and it is represented by $$(x_1,\cdots,x_m) \rightarrow (x_1,\cdots,x_{n-1},{\Sigma}_{j=1}^{m-n+1} {x_{n+j-1}}^2)$$ via these local coordinates.
  \end{Def}
  
  \begin{Prop}
  	In Definition \ref{def:1}, the singular set of the map is an {\rm (}$n-1${\rm )} smooth closed submanifold of the manifold of the domain and has no boundary. Furthermore, the restriction of the special generic map to the singular set is a smooth immersion.
  \end{Prop}
  
Unit spheres are regarded as simplest smooth closed manifolds. They are topologically spheres. In considerable cases, smooth manifolds which are homeomorphic to unit spheres and not diffeomorphic to ones admit no special generic maps in considerable situations. See \cite{calabi,saeki,saeki2,wrazidlo}.

More generally, the topologies and the differentiable structures of manifolds admitting such maps are strongly restricted in general. \cite{nishioka,saeki,saeki2,saekisakuma,saekisakuma2,sakuma,wrazidlo3} explicitly show this.

On the other hands, special generic maps seem to represent elementary and fundamental manifolds naturally in some senses. Manifolds represented as connected sums of products of manifolds diffeomorphic to unit spheres and more general manifolds seeming to be regarded as such manifolds in geometry admit natural special generic maps in suitable situations.

This means explicitly that special generic maps are attractive in singularity theory of differentiable maps and applications to algebraic topology and differential topology of manifolds. 

Our paper mainly concerns some studies of the author on restrictions on the cohomology rings of the manifolds admitting such maps. Related studies are \cite{kitazawa1,kitazawa2,kitazawa3,kitazawa4,kitazawa5,kitazawa6,kitazawa7} for example. Studies on the homology groups are in the articles referred to before. They are presented in the third section.

We introduce Main theorems. We need some fundamental terminologies, notions and notation on algeberaic topology such as homology groups, cohomology groups and cohomology rings. We also expect readers have sufficient knowledge on them and we review later shortly.
For an integer $k>0$, ${\mathbb{R}}^k$ denotes the $k$-dimensional Euclidean space, which is regarded as a smooth manifold canonically. ${\mathbb{R}}^1$ is also denoted by $\mathbb{R}$, containing the ring of all rational numbers as a subring, denoted by $\mathbb{Q}$, and the ring of all integers, denoted by $\mathbb{Z} \subset \mathbb{Q}$. 

\begin{MainThm}
	\label{mthm:1}
	Let $k$ be an arbitrary integer.
	Let $A$ be the quotient ring of the form $\mathbb{Z}/k\mathbb{Z}$ where $k\mathbb{Z}$ is the set of all integers of the form $ki$ {\rm (}$i \in \mathbb{Z}${\rm )}. 
	Let $m \geq 7$ be an integer and $M$ an $m$-dimensional closed and simply-connected manifold. Assume also the existence of elements $u_{0,1}$ and $u_{0,2}$ of the 2nd cohomology group $H^2(M;A)$ enjoying the following properties.
	\begin{itemize}
		\item The cup product $u_{0,1} \cup u_{0,2}$ is not the zero element of $H^4(M;A)$.
		\item We can choose no element of $u_0 \in H^4(M;\mathbb{Z})$ enjoying both the following properties.
		\begin{itemize}
			\item
			The order of $u_0$ is infinite.
			\item The cup product $u_{0,1} \cup u_{0,2}$ is obtained by changing the coefficient ring of the element $u_0$ using the natural quotient map from $\mathbb{Z}$ onto $A$ canonically. 
		\end{itemize}	
	\end{itemize}
	
	Then $M$ admits no special generic maps into ${\mathbb{R}}^n$ for $n=1,2,3,4,5$. Furthermore, we cannot drop the condition $m>6$ here.
\end{MainThm}
\begin{MainThm}
	\label{mthm:2}
	Let $k>0$ be an arbitrary positive integer.
Let $A$ be the quotient ring of the form $\mathbb{Z}/k\mathbb{Z}$ where $k\mathbb{Z}$ is the set of all integers of the form $ki$ {\rm (}$i \in \mathbb{Z}${\rm )}. 
Let $M$ be a $6$-dimensional closed and simply-connected manifold. Assume also the existence of elements $u_{0,1}$ and $u_{0,2}$ of the 2nd cohomology group $H^2(M;A)$ enjoying the following properties.
\begin{itemize}
	\item The cup product $u_{0,1} \cup u_{0,2}$ is not the zero element of $H^4(M;A)$.
	\item $u_{0,i}$ is not obtained by changing the coefficient ring of any element $u_{i,\mathbb{Z}}$ using the natural quotient map from $\mathbb{Z}$ onto $A$ canonically. 
	\item We can choose no element of $u_0 \in H^2(M;\mathbb{Z})$ enjoying both the following properties.
	\begin{itemize}
		\item
		The order of $u_0$ is infinite.
		\item The cup product $u_{0,1} \cup u_{0,2}$ is obtained by changing the coefficient ring of the element $u_0$ using the natural quotient map from $\mathbb{Z}$ onto $A$ canonically. 
	\end{itemize}	
\end{itemize}

Then $M$ admits no special generic maps into ${\mathbb{R}}^n$ for $n=1,2,3,4,5$.
\end{MainThm}
\begin{MainThm}
	\label{mthm:3}
	
	Let $m \geq 7$ be an integer and $M$ an $m$-dimensional closed and simply-connected manifold. Assume also the existence of a family $\{u_j\}_{j=1}^l$ of $l>0$ elements of the 2nd cohomology group $H^4(M;\mathbb{Z})$ enjoying the following properties.
	\begin{itemize}
		\item For some integer $k_j$ and a commutative ring $A_j$ isomorphic to the quotient ring of the form $\mathbb{Z}/k_j\mathbb{Z}$ where $k_j\mathbb{Z}$ is as in Main Theorem \ref{mthm:1}, 
		there exist two elements $u_{j,1}, u_{j,2} \in H^2(M;A_j)$ and the cup product $u_{j,1} \cup u_{j,2} \in H^4(M;A)$ is obtained from $u_j$ by changing the coefficient ring using the canonical quotient map from $\mathbb{Z}$ onto $A_j$ and not the zero element.. 
		\item The elements in $\{u_j\}_{j=1}^l$ are mutually independent{\rm :} their orders are infinite of course. 
	\end{itemize}
	
	If $M$ admits a special generic map $f:M \rightarrow {\mathbb{R}}^5$, then the number of the singular set of $f$ must be at least $l+1$.
\end{MainThm}
\begin{MainThm}
	\label{mthm:4}
	Let $m \geq 7$ be an arbitrary integer. We have the following two.
	\begin{enumerate}
		\item $S^2 \times S^2 \times S^{m-4}$ admits a special generic map into ${\mathbb{R}}^5$ and the minimum of the numbers of connected components of the singular sets of special generic maps there into ${\mathbb{R}}^5$ there is $2$. 
		\item A manifold represented as a connected sum of $l_0>1$ copies of $S^2 \times S^2 \times S^{m-4}$ considered in the smooth category admits a special generic map into ${\mathbb{R}}^5$ and the minimum of the numbers of connected components of the singular sets of special generic maps there into ${\mathbb{R}}^5$ is $l_0+1$. 
		\end{enumerate}
\end{MainThm}

The organization of our paper is as follows. The next section is for preliminaries. We introduce fundamental terminologies, notions and notation on smooth manifolds, diffeomorphisms and bundles and algebraic topology for example. We also review fundamental properties of special generic maps. As an advanced exposition, in the third section, we introduce important results on special generic maps, mainly ones presented shortly before. 
Mainly, we concentrate on special generic maps on closed and simply-connected manifolds whose dimensions are at least $6$ in our paper. We add a short general exposition on closed and simply-connected manifolds, which are important objects in algebraic topology and differential topology of manifolds. The fourth section is devoted to arguments on Main Theorems. We prove Main Theorems and compare to results presented in the third sections.

\section{Preliminaries on smooth manifolds, diffeomorphisms and bundles, fundamental algebraic topology and special generic maps.}

\subsection{Smooth manifolds, diffeomorphisms and bundles.}
${\mathbb{R}}^k$ is also a Riemannian manifold endowed with the standard Euclidean
 metric. $||x|| \in {\mathbb{R}}^k$ denotes the distance between $x \in {\mathbb{R}}^k$ and the origin $0 \in {\mathbb{R}}^k$ where the distance is induced canonically from the metric. $S^k:=\{x \in {\mathbb{R}}^{k+1} \mid ||x||=1\}$ denotes the {\it $k$-dimensional unit sphere} for an integer $k \geq 0$. This is also a closed smooth submanifold with no boundary. It is homeomorphic to a sphere and connected for $k \geq 1$ and it is a two-point set with the discrete topology for $k=0$.
 $D^k:=\{x \in {\mathbb{R}}^{k} \mid ||x||=1\}$ denotes the {\it $k$-dimensional unit disk} for an integer $k \geq 1$. This is also a closed smooth submanifold and connected.
 
  A {\it diffeomorphism} means a smooth map with no singular points which is also a homeomorphism. A {\it diffeomorphism on a smooth manifold} means a diffeomorphism from the manifold onto the manifold itself.
 The {\it diffeomorphism group} of a smooth manifold is defined as the group of all diffeomorphisms on the manifolds. We regard this as the topological space endowed with the {\it Whitney $C^{\infty}$} topology. For this topology, see \cite{golubitskyguillemin} for example. Two smooth manifolds are {\it diffeomorphic} if there exists a diffeomorphism from a manifold to the other manifold and this gives an equivalence relation on the class of all smooth manifolds. 

We call a closed and connected smooth manifold which is homeomorphic to a unit sphere a {\it homotopy sphere}. We say that a homotopy sphere which is (not) diffeomorphic to a unit sphere a {\it standard} (resp. an {\it exotic}) sphere. 
 
A {\it smooth} bundle means a bundle whose fiber is a smooth manifold and whose structure group is the diffeomorphism group of the manifold. A {\it linear} bundle means a bundle whose fiber is a Euclidean space, a unit sphere, or a unit disk and whose structure group consists of all linear transformations, where we (can) define linear transformations canonically. 

\subsection{Fundamental algebraic topology.}
For algebraic topology, see also \cite{hatcher}.

Let $A$ be a commutative ring.
 Let $(X,X^{\prime})$ be a pair of topological spaces satisfying $X^{\prime} \subset X$ where the spaces may be empty.
$H_j(X,X^{\prime};A)$ (resp. $H^j(X,X^{\prime};A)$) denotes the {\it $j$-th} ({\it co}){\it homology group} of the pair $(X,X^{\prime})$ whose {\it coefficient ring} is $A$. If $X^{\prime}$ is empty, then "the pair" is changed into the space $X$ and we use $H_j(X;A)$ (resp. $H^j(X;A)$) instead.

For a topological space $X$, we consider the direct sum $H^{\ast}(X;A):={\oplus}_j H^j(X;A)$ for all integers. This has the structure of a graded commutative algebra over $A$ where the product ${\cup}_{j=1}^l u_j$ of a sequence $\{u_j\}_{j=1}^l \subset H^{\ast}(X;A)$ is defined as the {\it cup product} of the sequence. In the case $l=2$, we also use $u_1 \cup u_2$ for the cup product of the ordered pair. This graded algebra is the {\it cohomology ring} of $X$ whose {\it coefficient ring} is $A$.

We can also define the {\it $j$-th homotopy group} ${\pi}_j(X)$ for a topological space $X$ of some suitable class. 
For example, the class of topological spaces homeomorphic to connected cell complexes satisfy this. 
For $j=1$, this is also the {\it fundamental group} of $X$.

For pairs of topological spaces $(X_1,{X_1}^{\prime})$ and $(X_2,{X_2}^{\prime})$ satisfying the relation ${X_j}^{\prime} \subset X_j$ where the spaces may be empty, we consider a continuous map $c:X_1 \rightarrow X_2$ mapping the subspace ${X_1}^{\prime}$ into ${X_2}^{\prime}$. We have the canonical homomorphisms such as ${c}_{\ast}:H_j(X_1,{X_1}^{\prime};A) \rightarrow H_j(X_2,{X_2}^{\prime};A)$, ${c}^{\ast}:H^j(X_2,{X_2}^{\prime};A) \rightarrow H^j(X_2,{X_2}^{\prime};A)$,
${c}^{\ast}:H^{\ast}(X_2;A) \rightarrow H^{\ast}(X_2;A)$ and ${c}_{\ast}:{\pi}_j(X_1) \rightarrow {\pi}_j(X_2)$.

Hereafter, for a manifold $X$, or more generally, a CW complex $X$, ${\dim X}$ denotes the dimension. 

For a compact, connected and oriented manifold $X$ whose boundary $\partial X$ may not be empty, we have a unique generator of $H_{\dim X}(X,\partial X;A)$, which is a free module of rank $1$. This is compatible with the orientation is defined as the fundamental class of $X$. 

For a manifold $Y$, for an element $u \in H_j(Y;A)$, suppose that there exist a compact, connected and orientable manifold $X$ of dimension $j$ and an embedding $i_X$ satisfying some conditions making the embedding not being so wild and mapping the boundary to the boudnary and the inteiror to the inteirior, ${i_X}_{\ast}(u_X)=u$ for the fundamental class $u_X$ of a suitably oriented X. Then $u$ is said to be {\it represented} by $X$. We note that in the case where the category we consider is the PL categoty, (equivalently the piecewise smooth), and the smooth category, for example, we do not need the condition to avoid the wildness of the embedding. Poincar\'e duality theorem for compact, connected and oriented manifolds is important. More precisely, for such a manifold $X$ whose boundary $\partial X$ may not be empty, the {\it Poincar\'e duals} to elements of $H_j(X,\partial X;A)$, $H^j(X,\partial X;A)$, $H_j(X;A)$ and $H^{j}(X;A)$ are uniquely defined as elements of $H^{\dim X-j}(X;A)$, $H^{\dim X-j}(X;A)$, $H^{\dim X-j}(X,\partial X;A)$ and $H_{\dim X-j}(X,\partial X;A)$ uniquely. 

For fundamental classes, Poincar\'e duality and related notions and theory here, we do not need the orientations of the manifolds in the case $A$ consists of elements whose orders are at most $2$.

We define the {\it cohomology dual} $u^{\ast} \in H^j(X,X^{\prime};A)$ to an element of $u$ in a basis of the submodule generated by all elements whose orders are infinite of $H_j(X,X^{\prime};A)$ consisting of elements which can not be divided by elements other than units
 where $X$ and $X^{\prime}$ are topological spaces satisfying $X^{\prime} \subset X$ and allowing themselves to be empty. We can define the cohomology dual to this given element uniquely.

\subsection{Fundamental properties of special generic maps.}
\begin{Prop}[\cite{saeki}]
	\label{prop:2}
	Let $m \geq n \geq 1$ be integers.
	\begin{enumerate}
		\item
		\label{prop:2.1}
		 Let $f:M \rightarrow N$ be a special generic map from an $m$-dimensional closed and connected manifold $M$ into an $n$-dimensional manifold $N$ with no boundary. Then we obtain the following objects and properties.
		\begin{enumerate}
			\item An $n$-dimensional compact and connected smooth manifold $W_f$.
			\item A smooth surjection $q_f:M \rightarrow W_f$.
			\item A smooth immersion $\bar{f}:W_f \rightarrow N$ enjoying the relation $f=\bar{f} \circ q_f$.

		\item Furthermore, 
		there exists a collar neighborhood $N(\partial W_f)$ of the boundary $\partial W_f \subset W_f$ and the composition of $f {\mid}_{f^{-1}(N(\partial W_f))}$ with the canonical projection to $\partial W_f$ gives a linear bundle over $\partial W_f$ whose fiber is the {\rm (}$m-n+1${\rm )}-dimensional unit disk $D^{m-n+1}$ where the manifold of the target of the former map is restricted to $N(\partial W_f)$.
		\item In addition, the restriction of $f$ to the preimage of the complementary set of the interior ${\rm Int}\ N(\partial W_f) \subset W_f$ gives a smooth bundle over the complementary set of ${\rm Int}\ N(\partial W_f) \subset W_f$ whose fiber is an {\rm (}$m-n${\rm )}-dimensional standard sphere for $m>n$ and diffeomorphic to the $0$-dimensional unit sphere $S^0$ for $m-n=0$. 
	\end{enumerate}
\item
\label{prop:2.2}
 Conversely, suppose that a smooth immersion of a compact and connected manifold $\bar{N}$ and a smooth immersion $\bar{f_N}$ of this manifold into an $n$-dimensional manifold $N$ with no boundary are given. Then we have an $m$-dimensional closed and connected manifold $M_0$ and a special generic map $f_0:M_0 \rightarrow N$ satisfying properties similar to ones with the triviality of the two bundles presented in {\rm (}\ref{prop:2.1}{\rm )}. In other words, we have the following objects and properties as before.  
 
 	\begin{enumerate}
 	\item $\bar{N}$ is identified with $W_{f_0}$ where "$f$ in $W_f$" is replaced by $f_0$ in the notation.
 	\item A smooth surjection $q_{f_0}:M \rightarrow W_{f_0}$.
 	\item $\bar{f_N}$ is identified with the immersion $\bar{f_0}:W_{f_0} \rightarrow N$ where "$f$ in $W_f$" and "$f$ in $\bar{f}$" are replaced by $f_0$ in the notation as before.
 	
 	\item Furthermore, 
 	there exists a collar neighborhood $N(\partial W_{f_0})$ of the boundary $\partial W_{f_0} \subset W_{f_0}$ and the composition of $f_0 {\mid}_{{f_0}^{-1}(N(\partial W_{f_0}))}$ with the canonical projection to $\partial W_{f_0}$ gives a trivial linear bundle over $\partial W_{f_0}$ whose fiber is the {\rm (}$m-n+1${\rm )}-dimensional unit disk $D^{m-n+1}$ where the manifold of the target of the former map is restricted to $N(\partial W_{f_0})$.
 	\item In addition, the restriction of $f_0$ to the preimage of the complementary set of the interior ${\rm Int}\ N(\partial W_{f_0}) \subset W_{f_0}$ gives a trivial smooth bundle over the complementary set of ${\rm Int}\ N(\partial W_{f_0}) \subset W_{f_0}$ whose fiber is an {\rm (}$m-n${\rm )}-dimensional standard sphere for $m>n$ and diffeomorphic to the $0$-dimensional unit sphere $S^0$ for $m-n=0$.
 	\item If $N$ is orientable or more generally $\bar{N}$ is orientable, then $M_0$ is constructed as an orientable manifold.
 \end{enumerate}
\end{enumerate}
\end{Prop}

Hereafter, connected sums and boundary connected sums of manifolds are considered in the smooth category unless otherwise stated.
\begin{Ex}
	\label{ex:1}
\begin{enumerate}
	\item \label{ex:1.1}
	 For an arbitrary canonical projection $f$ of an arbitrary unit sphere, $W_f$ is regarded as the unit disk in the Euclidean space of the target.
	\item \label{ex:1.2}
	 Let $m \geq n \geq 2$ be integers.
	Consider an arbitrary $m$-diemnsional closed and connected manifold $M$ represented as a connected sum of products of two standard spheres where at least one of each pair of the standard spheres is smaller than $n$. Let the sequence of the products denoted by $\{S^{n_j} \times S^{m-n_j}\}_{j=1}^l$ where the sequence is denoted as a sequence of products of unit spheres of length $l>0$ with $n_j<n$. $M$ admits a special generic map $f:M \rightarrow {\mathbb{R}}^n$ enjoying the following three.
	\begin{enumerate}
		\item $W_f$ in Proposition \ref{prop:2} is a smoothly embedded $n$-dimensional manifold diffeomorphic to one represented as a boundary connected sum of $l$ manifolds in the sequence $\{S^{n_j} \times D^{n-n_j}\}_{j=1}^l$.
		\item The map $\bar{f}$ in Proposition \ref{prop:2} is an embedding.
		\item $f$ is constructed as a map like $f_0$ in Proposition \ref{prop:2} (\ref{prop:2.2}).
		\end{enumerate}  
\end{enumerate}
\end{Ex}
We can check Example \ref{ex:1} by investigating \cite{saeki} for example. This is also regarded as a fundamental exercise on theory of Morse functions and special generic maps. To discuss this more rigorously, we may need some important technique on singularity theory of differentiable maps, including fundamental theory of Morse functions and some related theory on (smooth) manifolds. 

In addition, hereafter, we also need such theory and related theory on differential topology to understand arguments rigorously.

 For related technique, consult \cite{golubitskyguillemin, milnor,milnor2} for example. 
\section{Existing advanced studies on special generic maps.}

\begin{Prop}[E.g. \cite{kitazawa0.1,kitazawa0.2,kitazawa0.3,saeki,saekisuzuoka}]
	\label{prop:3}
	Let $m > n \geq 1$ be integers. Let $A$ be a commutative ring.
	Let $f:M \rightarrow N$ be a special generic map from an $m$-dimensional closed and connected manifold $M$ into an $n$-dimensional manifold $N$ with no boundary. Then $q_f$ in Proposition \ref{prop:2} induces the following isomorphisms of groups for an arbitrary integer $0 \leq j \leq m-n$.
	\begin{enumerate}
		\item ${q_f}_{\ast}:H_j(M;A) \rightarrow H_j(W_f;A)$. 
		\item ${q_f}^{\ast}:H^j(W_f;A) \rightarrow H^j(M;A)$.
		\item ${q_f}_{\ast}:{\pi}_j(M;A) \rightarrow {\pi}_j(W_f;A)$. 
	\end{enumerate}
\end{Prop}
Hereafter, it is an important fact that smooth manifolds are regarded as PL (or equivalently, piecewise smooth manifolds) in a canonical way and we regard smooth manifolds as such PL manifolds if we need. If a manifold is not smooth, PL, nor piecewise smooth, then it is well known to be regarded as a CW complex.

On the PL category or equivalently, the piecewise smooth category,
we need notions such as {\it PL homeomorphisms}, {\it piecewise smooth homeomorphisms} and two {\it PL homeomorphic} manifolds, which we can guess the definitions easily. For objects in these categories, CW complexes and more generally, cell complexes, the notion of {\it collapsing} or the notion that a complex {\it collapses} to another complex is defined. We need this notion in Proposition \ref{prop:4}. However, we can guess the definition easily or we can consult \cite{hudson} to know the definition for example. 

Hereafter, we may need some arguments on the PL category. For the PL category, \cite{hudson}, presented just before, is one of books presenting fundamental theory of the PL category systematically and see also \cite{bryant} for example.

It may regarded as a good exercise to prove Proposition \ref{prop:3} by using the following proposition and consulting the articles and books being referred to here. 

\begin{Prop}[E.g. \cite{kitazawa0.2,kitazawa0.3,saekisuzuoka}]
\label{prop:4}	
	In Proposition \ref{prop:3}, we have an {\rm (}$m+1${\rm )} compact and connected manifold $W$ in the topology {\rm (}PL or piecewise smooth{\rm )} category enjoying the following properties.
	\begin{enumerate}
		\item \label{prop:4.1}
		 $M$ is the boundary $\partial W \subset W$ as  a PL or piecewise smooth manifold if we consider the PL or piecewise smooth category.     
		\item \label{prop:4.2}
		$W$ collapses to $W_f$ where $W_f$ is embedded in $W$ in a suitable way. Furthermore, $q_f$ is regarded as the composition of the inclusion $i_M:M \rightarrow W$ with a continuous {\rm (}resp. PL or piecewise smooth{\rm )} map $r:W \rightarrow W_f$ giving a collapsing.
		\item \label{prop:4.3}
		Moreover, $r$ gives a bundle over $W_f$ whose fiber is {\rm PL} homeomorphic to the {\rm (}$m-n+2${\rm )}-dimensional unit disk $D^{m-n+1}$ and whose structure group consists of PL homeomorphisms in the case where we consider the PL category. 
		\item \label{prop:4.4}
		 We have the following relation between $r$ and the restriction of $r$ to the boundary $M=\partial W$ where we abuse the notation in Proposition \ref{prop:2}.
		\begin{enumerate}
			\item \label{prop:4.4.1}
			 The restriction of $r$ to $r^{-1}(W_f-{\rm Int}\ N(\partial W_f))$ is regarded as the projection of a bundle whose fiber is PL homeomorphic to the {\rm (}$m-n+1${\rm )}-dimensional unit disk $D^{m-n+1}$ and whose structure group consists of piecewise smooth diffeomorphisms in the case where we consider the PL or the piecewise smooth category. The restriction of this to the intersection $r^{-1}(W_f-{\rm Int}\ N(\partial W_f)) \bigcap M$ is the original projection of the smooth bundle over $W_f-{\rm Int}\ N(\partial W_f)$. This original bundle is also a subbundle of the previous bundle whose fiber is $\partial D^{m-n+1} \subset D^{m-n+1}$.  
			\item \label{prop:4.4.2}
			 The composition of the restriction of $r$ to $r^{-1}(N(\partial W_f))$ with the canonical projection to $\partial W_f$ is regarded as the projection of a bundle whose fiber is PL homeomorphic to the {\rm (}$m-n+2${\rm )}-dimensional unit disk $D^{m-n+2}$ and whose structure group consists of piecewise smooth diffeomorphisms in the case where we consider the PL or the piecewise smooth cateogory. Note that the manifold of the target of the restriction of $r$ to $r^{-1}(N(\partial W_f))$ is restricted to $N(\partial W_f)$.
			
			 Furthermore, we explain about the restriction of this projection to the intersection $r^{-1}(N(\partial W_f)) \bigcap M$ as {\rm (}\ref{prop:4.4.1}{\rm )}. This is the original projection of the original linear bundle over $N(\partial W_f)$ in Proposition \ref{prop:2}. This original bundle is also a subbundle of the previous bundle whose fiber is PL homeomorphic to $D^{m-n+2}$ and the fiber of this subbundle is regarded as $D^{m-n+1} \subset \partial D^{m-n+2} \subset D^{m-n+2}$ where $D^{m-n+1}$ is embedded as a hemisphere of the unit sphere $S^{m-n+1}=\partial D^{m-n+2}$.  
				\end{enumerate}
	\end{enumerate}
	
\end{Prop}
\begin{Thm}
	\label{thm:1}
	Let $m>0$ be a positive integer.
	\begin{enumerate}
		\item \label{thm:1.1}
		 {\rm (}Reeb's theorem{\rm )}
		An $m$-dimensional closed and connected manifold $M$ admits a special generic map into $\mathbb{R}$ if and only if $M$ is either of the following two.
		\begin{enumerate}
			\item An $m$-dimensional homotopy sphere where $m \geq 1$ and $m \neq 4$.
			\item A $4$-dimensional standard sphere.
		\end{enumerate}
		\item \label{thm:1.2}
		 {\rm (}\cite{saeki}{\rm )}
		Let $m \geq 2$. An $m$-dimensional closed and connected manifold $M$ admits a special generic map into ${\mathbb{R}}^2$ if and only if $M$ is either of the following four.
		\begin{enumerate}
			\item An $m$-dimensional homotopy sphere where $m \geq 2$ and $m \neq 4$.
			\item A $4$-dimensional standard sphere.
			\item An $m$-dimensional manifold represented as a connected sum of finitely many manifolds diffeomorphic to the total spaces of smooth bundles over $S^1$ whose fibers are homotopy spheres where $m \neq 5$.
			\item A $5$-dimensional manifold represented as a connected sum of finitely many manifolds diffeomorphic to the total spaces of smooth bundles over $S^1$ whose fibers are $4$-dimensional standard spheres.
		\end{enumerate}
	\item \label{thm:1.3}
	 {\rm (}\cite{saeki}{\rm )}
	Let $n<m$ be a positive integer. 
	An $m$-dimensional closed and connected manifold $M$ admitting a special generic map $f$ into ${\mathbb{R}}^n$ 
	is a homotopy sphere if and only if $W_f$ is contractible.
	
	\item \label{thm:1.4}
	 {\rm (}\cite{calabi,saeki,saeki2}{\rm )}
	Let $m \geq 4$. An $m$-dimensional homotopy sphere $M$ admitting a special generic map $f$ into ${\mathbb{R}}^n$ with $m-n=1,2,3$ must be a standard sphere.
	\item \label{thm:1.5}
	 {\rm (}\cite{wrazidlo}{\rm )}
	Among all 28 types of $7$-dimensional oriented homotopy spheres, discussed in \cite{milnor} and later in \cite{eellskuiper} and \cite{kervairemilnor} for example, at least ones of 14 types admit no special generic maps into ${\mathbb{R}}^3$.  
	
		\item \label{thm:1.6}
		 {\rm (}\cite{saeki}{\rm )}
		Let $m \geq 4$. An $m$-dimensional closed and simply-connected manifold $M$ admitting a special generic map into ${\mathbb{R}}^3$ must be either of the following two.
	\begin{enumerate}
		\item An $m$-dimensional homotopy sphere which is not a $4$-dimensional exotic sphere.
		\item An $m$-dimensional manifold represented as a connected sum of finitely many manifolds diffeomorphic to the total spaces of smooth bundles over $S^2$ whose fibers are homotopy spheres. Furthermore, the {\rm (}$m-2${\rm )}-dimensional homotopy spheres here are not $4$-dimensional exotic spheres.
		\end{enumerate}
		\item \label{thm:1.7}
		 {\rm (}\cite{saeki}{\rm )}
		Let $m=4,5,6$. An $m$-dimensional closed and simply-connected manifold $M$ admits a special generic map into ${\mathbb{R}}^3$ if and only if $M$ is either of the following two.
	\begin{enumerate}
		\item An $m$-dimensional standard sphere.
		\item An $m$-dimensional manifold represented as a connected sum of finitely many manifolds diffeomorphic to the total spaces of linear bundles over $S^2$ whose fibers are diffeomorphic to the {\rm (}$m-2${\rm )}-dimensional unit sphere $S^{m-2}$.

	\end{enumerate}

	\item \label{thm:1.8}
	 {\rm (}\cite{nishioka} {\rm (}the case $m=5${\rm )} and \cite{kitazawa6} {\rm (}the case $m=6${\rm )}{\rm )}
	Let $m=5,6$. An $m$-dimensional closed and simply-connected manifold $M$ admits a special generic map into ${\mathbb{R}}^4$ if and only if $M$ is either of the following two.
\begin{enumerate}
	\item An $m$-dimensional standard sphere.
	\item An $m$-dimensional manifold represented as a connected sum of finitely many manifolds diffeomorphic to the total spaces of linear bundles over $S^2$ whose fibers are diffeomorphic to the {\rm (}$m-2${\rm )}-dimensional unit sphere $S^{m-2}$ or ones over $S^3$ whose fibers are diffeomorphic to the {\rm (}$m-3${\rm )}-dimensional unit sphere $S^{m-3}$.
\end{enumerate}
\end{enumerate}
\end{Thm}
\begin{Thm}[\cite{kitazawa1}]
	\label{thm:2}
	Let $m>n \geq 1$ be integers. Let $A$ be a commutative ring.
	For an $m$-dimensional closed and connected manifold $M$, suppose the existence of a sequence $\{u_j\}_{j=1}^{l} \subset H^{\ast}(M;A)$ of length $l>0$ satisfying the following conditions.
	\begin{itemize}
		\item The degree of each element is at most $m-n$.
		\item The sum of the degrees of these $l$ elements is at least $n$.
		\item The cup product ${\cup}_{j=1}^l u_j$ is not the zero element of $H^{\ast}(M;A)$.
	\end{itemize}
Then $M$ admits no special generic maps into any connected manifold $N$ which is not compact and has no boundary.
\end{Thm}
\begin{Cor}
	\begin{enumerate}
		\item Let $m \geq 2$ be an integer. The $m$-dimensional real projective space admits no special generic maps into ${\mathbb{R}}^n$ for $1 \leq n \leq m-1$.
\item Let $m \geq 2$ be an integer. The $m$-dimensional complex projective space, which is an $m$-dimensional complex manifold and a $2m$-dimensional closed and simply-connected manifold, admits no special generic maps into ${\mathbb{R}}^n$ for $1 \leq n \leq 2m-2$.
	\end{enumerate}
\end{Cor}
In the case of the $7$-dimensional real projective space, see also \cite{wrazidlo2}.
\begin{Thm}[\cite{kitazawa6}]
	\label{thm:3}
	Let $m \geq 2$ be an integer. The $m$-dimensional complex projective space admits no special generic maps into ${\mathbb{R}}^n$ for $1 \leq n \leq 2m$.
\end{Thm}
\begin{Thm}[\cite{kitazawa3}]
	\label{thm:4}
	Let $m \geq 7$ be an integer.
	For an $m$-dimensional closed and simply-connected manifold $M$, suppose the existence of an element $u \in H^2(M;\mathbb{Z})$ such that the cup product $u \cup u \in H^4(M;\mathbb{Z})$ is not divisible by $2$. Then $M$ admits no special generic maps into ${\mathbb{R}}^5$. Furthermore, in this situation, we cannot drop the condition $m>6$. 
\end{Thm}
\begin{Thm}[E.g. \cite{kitazawa4,kitazawa5}]
	\label{thm:5}
		Let $m \geq 6$ be an integer.
	For an $m$-dimensional closed and simply-connected manifold $M$, suppose the existence of a special generic map $f :M \rightarrow {\mathbb{R}}^5$ whose singular set is connected.
	
	Then for any commutative ring isomorphic to the quotient ring $A:=\mathbb{Z}/k\mathbb{Z}$ of $\mathbb{Z}$ where $k\mathbb{Z}$ is as presented in Main Theorem \ref{mthm:1}, the cup product of any two ordered pair of elements of $H^2(M;A)$ is always the zero element of $H^4(M;A)$. 
\end{Thm}

Most of existing results introduced here are on special generic maps on closed and simply-connected manifolds.

Closed and simply-connected manifolds whose dimensions are at least $5$ are central objects in classical algebraic topology and differential topology. Their classifications are due to sophisticated ways and related results are via algebraic and abstract tools. This is due to the situation that the dimensions are sufficiently high. $4$-dimensional cases are difficult due to the situation that the dimensions are not sufficiently high and difficulty in understanding differentiable structures are well-known whereas some works such as Freedman's ones have contributed to topological classifications for example. The $3$-dimensional case is so-called Poincar\'e's problem, solved by Perelman. 

This is also important that classifications of $5$, $6$ and $7$-dimensional closed and simply-connected manifolds via more explicit algebraic systems are known. \cite{barden} is for $5$-dimensional ones. This is a key fact in the work of Nishioka \cite{nishioka} in Theorem \ref{thm:1} (\ref{thm:1.8}).
\cite{jupp,wall,zhubr,zhubr2} are on $6$-dimensional cases. Some of them and Theorem \ref{thm:2} are key ingredients in studies in \cite{kitazawa6} of the author, one of which is presented in Theorem \ref{thm:1} (\ref{thm:1.8}). For $7$-dimensional cases, \cite{kreck} is one of worls on such classifications.
\section{On Main Theorems.}
We discuss Main Theorems and related explicit theory.
\subsection{Some important propositions} 
We prove some important propositions. Some ingredients are due to methods in \cite{kitazawa3,kitazawa4,kitazawa5,kitazawa6} for example.
\begin{Prop}
	\label{prop:5}
	Let $m>n \geq 1$ be integers. Let $M$ be an $m$-dimensional closed, connected and orientable manifold. Let $l$ denote the number of connected components of the singular set of $f$ and suppose that $l>1$. Then there exist exactly $l-1$ elements of $H_{n-1}(M;\mathbb{Z})$ enjoying the following properties.
	\begin{enumerate}
		\item Each of these $l-1$ elements is represented by some connected component of the singular set of $f$.
		\item Each of these $l-1$ elements is not divisible by any integer greater than $1$.
		\item These $l-1$ elements form a basis of a submodule of $H_{n-1}(M;\mathbb{Z})$. 
		\end{enumerate}
\end{Prop}
\begin{proof}
	In the proof, we abuse the notation in Proposition \ref{prop:2}.
	
	Choose an arbitrary connected component $C_0$ of the singular set of $f$. For the remaining $l-1$ connected components, we have a family $\{L_j\}$ of $l-1$ $1$-dimensional compact and connected submanifolds of $W_f$ enjoying the following properties.
	\begin{itemize}
		\item For each connected component $C_j$ of the $l-1$ connected components of the singular set of $f$, $L_j$ is diffeomorphic to a closed interval. The boundary of $L_j$, consisting of exactly two points, are smoothly embedded in the boundary $\partial W_f$ and one point is embedded in $C_0$ and the other point is in $C_j$. The interior is embedded in the interior ${\rm Int}\ W_f \subset W_f$.  
		\item The restriction of $q_f$ to the preimage ${q_f}^{-1}(L_j)$ is regarded as the domain of a Morse function on an ($m-n+1$)-dimensional homotopy sphere.
		\item $L_{j_1}$ and $L_{j_2}$ are mutually disjoint for $j_1 \neq j_2$. 
	\end{itemize}
Poincar\'e duality or intersection theory of closed submanifolds with no boundaries or elements of the homology and cohomology groups for $W_f$ and $M$ completes the proof. 
\end{proof}

For a finitely generated commutative group $G$, let ${\rm Fi}(G)$ denote the subgroup generated by all elements whose orders are finite. We have an internal direct sum decomposition of $G$ into ${\rm Fi}(G)$ and some free subgroup ${\rm Fr}(G)$. ${\rm Fr}(G)$ is not unique in general whereas the rank of ${\rm Fr}(G)$ is invariant. This is due to a well-known fundamental theorem on the structures of finitely generated commutative groups.

\begin{Prop}
	\label{prop:6}
	Let $n \geq 5$ be an integer and $m=n+1$.
	Let $M$ be an $m$-dimensional closed and simply-connected manifold.
	Let $f:M \rightarrow {\mathbb{R}}^n$ be a special generic map. Then ${q_f}_{\ast}$ maps ${\rm Fi}(H_2(M;\mathbb{Z}))$ isomorphically onto ${\rm Fi}(H_2(W_f;\mathbb{Z}))$ where we abuse the notation in Proposition \ref{prop:2}.
\end{Prop}
\begin{proof}
	$M$ is simply-connected and $W_f$ is also simply-connected by Proposition \ref{prop:3}. We have the relation $(n-1)+2=n+1=m$. Each element of ${\rm Fi}(H_2(M;\mathbb{Z}))$ is represented by a smoothly embedded $2$-dimensional standard sphere. From these facts, a fact on the elements represented by connected components of the singular set of $f$ in Proposition \ref{prop:5} and with fundamental arguments on intersection theory, we can take the $2$-dimensional sphere in $M$ apart from the singular set of $f$. Take this element of ${\rm Fi}(H_2(M;\mathbb{Z}))$ as an element which is not the zero element. If it is mapped to the zero element by the homomorphism ${q_f}_{\ast}$, then the element must be the zero element from the fact that $W_f$ is simply-connected and an elementary argument on homotopy theory and bundles. This is a contradiction and ${q_f}_{\ast}$ maps ${\rm Fi}(H_2(M;\mathbb{Z}))$ into ${\rm Fi}(H_2(W_f;\mathbb{Z}))$ by a monomorphism. Each element of ${\rm Fi}(H_2(W_f;\mathbb{Z}))$ is represented by a smoothly embedded $2$-dimensional standard sphere in the interior ${\rm Int}\ W_f \subset W_f$ such that the bundle given by the restriction of $f$ over the preimage of the $2$-dimensional sphere is a trivial linear bundle whose fiber is the unit sphere $S^1$. This is due to well-known classifications of bundles over paracompact spaces whose fibers are circles by the 2nd cohomology groups of the base spaces where the coefficient ring is $\mathbb{Z}$. We can take a section of the bundle over the $2$-dimensional sphere to complete the proof of the fact that the homomorphism we consider is an epimorphism.
	
	This completes the proof.
\end{proof}

For systematic theory of linear bundles and more general bundles, see \cite{milnorstasheff,steenrod} for example.

\begin{Rem}
	If $n \geq 4$, then ${\rm Fi}(H_2(W_f;\mathbb{Z}))$ must be the trivial group in Proposition \ref{prop:6} due to the proof of Theorem \ref{thm:1} (\ref{thm:1.8}) in \cite{nishioka}.
\end{Rem}
By the proof of Proposition \ref{prop:5}, we have the following proposition.
\begin{Prop}
	\label{prop:7}
	In the situation of Proposition \ref{prop:5} and its proof, choose an arbitrary connected component $C_0$ of the singular set of $f$. For the family $\{C_j\}$ of all of the $l-1$ remaining connected components of the singular set of $f$, we have a family $\{L_j\}$ of $1$-dimensional compact and connected manifolds diffeomorphic to a closed interval like one in the proof of Proposition \ref{prop:5} enjoying the following three.
\begin{itemize}
	\item The boundary of $L_j$, consisting of exactly two points, are smoothly embedded in the boundary $\partial W_f$ and one point is embedded in $C_0$ and the other point is in $C_j$. The interior is embedded in the interior ${\rm Int}\ W_f \subset W_f$.  
		\item The restriction of $q_f$ to the preimage ${q_f}^{-1}(L_j)$ is regarded as the manifold of the domain of a Morse function on an {\rm (}$m-n+1${\rm )}-dimensional homotopy sphere.
		\item $L_{j_1}$ and $L_{j_2}$ are mutually disjoint for $j_1 \neq j_2$. 
	\end{itemize}
Then we have the following bases of suitable subgroups of $H_1(W_f,\partial W_f;\mathbb{Z})$ and $H_2(M;\mathbb{Z})$.
\begin{itemize}
	\item A basis consisting of $l-1$ elements of $H_1(W_f,\partial W_f;\mathbb{Z})$ each element of which is not divisible by any integer greater than $1$ and is represented by some manifold $L_j$ in the family.
	\item A basis consisting of $l-1$ elements of $H_2(M;\mathbb{Z})$ each element of which is not divisible by any integer greater than $1$ and is represented by the manifold ${q_f}^{-1}(L_j)$ for some manifold $L_j$ in the family.
	\end{itemize}
Furthermore, suppose also that $H_1(W_f;\mathbb{Z})$ or $H_1(M;\mathbb{Z})$ is the trivial group. Then by Proposition \ref{prop:3}, these two groups are isomorphic{\rm :} for example assume that $M$ is simply-connected as an explicit case. Under this additional constraint, take another similar family $\{{L_j}^{\prime}\}$ of $1$-dimensional compact and connected manifolds smoothly embedded in $W_f$ instead of $\{L_j\}$. Then we have the same bases as ones obtained in the original case of $\{L_j\}$.
\end{Prop}
\begin{Prop}
	\label{prop:8}
	In the situation of Proposition \ref{prop:6}, assume that the number of the singular set of $f$ is $l${\rm :} note that $l>0$ holds since $M$ is compact and ${\mathbb{R}}^n$ is not compact.
	\begin{enumerate}
		\item
		\label{prop:8.1}
		 Assume also that ${\rm Fr}(H_2(W_f;\mathbb{Z}))$ is the trivial group. Then
	we have a family $\{L_j\}$ of $1$-dimensional compact and connected manifolds diffeomorphic to a closed interval
	like one in Propositions \ref{prop:5} and \ref{prop:7} and the proof and $H_2(M;\mathbb{Z})$ is generated by the disjoint union of the set of all elements represented by the $l-1$ $2$-dimensional spheres in the family $\{{q_f}^{-1}(L_j)\}$ like one in the proof of Propositions \ref{prop:5} and \ref{prop:7} and the finite group ${\rm Fi}(H_2(M;\mathbb{Z}))$. $H_2(M;\mathbb{Z})$ is also isomorphic to the direct sum of the free group of rank $l-1$ and ${\rm Fi}(H_2(W_f;\mathbb{Z}))$. 
	\item \label{prop:8.2}
	Assume also that $M$ is a manifold like $M_0$ in Proposition \ref{prop:2} {\rm (}\ref{prop:2.2}{\rm )} and that the map $f$ is a map like $f_0$ there. Then
	we have a family $\{L_j\}$ of $1$-dimensional compact and connected manifolds diffeomorphic to a closed interval
	like one in Propositions \ref{prop:5} and \ref{prop:7} and the proof and $H_2(M_0;\mathbb{Z})$ is generated by the disjoint union of the set of all elements represented by the $l-1$ $2$-dimensional spheres in the family $\{{q_f}^{-1}(L_j)\}$ like one in the proof of Propositions \ref{prop:5} and \ref{prop:7} and the group isomorphic to $H_2(W_{f_0};\mathbb{Z})$.
\end{enumerate}
\end{Prop}
\begin{proof}
We abuse the notation in Proposition \ref{prop:4}. For example, we need an ($m+1$)-dimensional compact and simply-connected manifold $W$, collapsing to and homotopy equivalent to a compact and simply-connected manifold $W_f$. We have a homology exact sequence for $(W,\partial W=M)$ as

$$ H_3(M;\mathbb{Z}) \rightarrow H_3(W;\mathbb{Z}) \rightarrow H_3(W,M;\mathbb{Z}) \rightarrow H_2(M;\mathbb{Z}) \rightarrow H_2(W;\mathbb{Z}) \rightarrow $$

\noindent and $H_3(W,M;\mathbb{Z})$ is isomorphic to $H^4(W;\mathbb{Z})$ and $H^4(W_f;\mathbb{Z})$ and by virtue of Poincar\'e duality theorem for the simply-connected manifold $W_f$, they are isomorphic to $H_1(W_f,\partial W_f,\mathbb{Z})$. 

We prove (\ref{prop:8.1}).
The ranks of $H_2(M;\mathbb{Z})$ and ${\rm Fr}(H_2(M;\mathbb{Z}))$ are $l-1$ by the homology exact sequence and Propositions \ref{prop:5} and \ref{prop:7}. Proposition \ref{prop:6} completes the proof.

We prove (\ref{prop:8.2}). By the triviality of the smooth bundles over the interior of $W_{f}=W_{f_0}$ and several facts in Proposition \ref{prop:4}, the homomorphism from $H_j(M;\mathbb{Z})$ into $H_j(W;\mathbb{Z})$ is surjective for any $j$. By respecting these facts, we have a homomorphism from $H_j(W;\mathbb{Z})$ into $H_j(M;\mathbb{Z})$ whose composition with the given homomorphism from $H_j(M;\mathbb{Z})$ into $H_j(W;\mathbb{Z})$ is the identity morphism.

This completes the proof.
\end{proof}
\subsection{Proofs of Main Theorems and additional expositions.}
\begin{proof}[A proof of Main Theorems \ref{mthm:1} and \ref{mthm:2}]
	Theorem \ref{thm:2} proves the case $n=1,2,3,4$.
	
	Suppose that there exists a special generic map $f:M \rightarrow {\mathbb{R}}^5$.
	
	We abuse the notation in Proposition \ref{prop:2} as before.
By Proposition \ref{prop:3}, $W_f$ is also simply-connected and the homomorphism ${q_f}_{\ast}:H_2(M;A) \rightarrow H_2(W_f;A)$ is an isomorphism. The Poincar\'e duals to ${q_f}^{\ast}(u_{0,1})$ and ${q_f}^{\ast}(u_{0,2})$ are elements of $H_3(W_f,\partial W_f;A)$ under a suitable orientation of $W_f$. 

We have elements of $H_5(W,\partial W;A)$ canonically by respecting the fiber of the bundle given by the map $r$, considering a kind of {\it prism-operators}, presented in \cite{hatcher} for example, or a kind of so-called {\it Thom-classes}, presented in \cite{milnorstasheff}, for example. By considering the boundary $\partial W=M \subset W$, we canonically have the elements of $H_4(M;A)$. See \cite{saeki} and see also \cite{kitazawa7} for related arguments. 

We can orient $M$, $W$ and $W_f$. The manifold $M$ is $m$-dimensional and the manifold $W_f$ is $5$-dimensional. We also have the relation $3+3-5=1$. By virtue of these expositions together with Poincar\'e duality or intersection theory for these manifolds, under suitable orientations of the manifolds, the Poincar\'e dual to the cup product $u_{0,1} \cup u_{0,2} \in H^4(M;A)$ is the sum of elements represented by the preimages of smoothly embedded circles in ${\rm Int}\ W_f$ or $1$-dimensional compact and connected manifolds in $\{L_j\}$ in Propositions \ref{prop:5} and \ref{prop:7}. Elements represented by the preimages of smoothly embedded circles in ${\rm Int}\ W_f$ are always the zero elements of $H_2(M;A)$. 
The order of an element of $H_2(M;\mathbb{Z})$ represented by
the preimage ${q_f}^{-1}(L_j)$ of some $L_j$ in the proof of Proposition \ref{prop:5} and Proposition \ref{prop:7} is always infinite.

This contradicts the assumption that we cannot find any element $u_0 \in H^{4}(M;\mathbb{Z})$ enjoying both the following properties.

\begin{itemize}
	\item The order of $u_0$ is infinite.
\item From $u_0$, by changing the coefficient ring via the natural quotient map from $\mathbb{Z}$ onto $A$ canonically, we have $u_{0,1} \cup u_{0,2} \in H^4(M;A)$.
\end{itemize}

This means that $M$ does not admit special generic maps into ${\mathbb{R}}^5$. This completes the proof of Main Theorem \ref{mthm:1} except the fact that the constraint $m \geq 7$ is optimal. 

We prepare a manifold diffeomorphic to $D_1:=S^2 \times S^2 \times D^1=S^2 \times S^2 \times [-1,1]$. 
Consider a homology class $e \in H_2(S^2 \times S^2 \times D^1;\mathbb{Z})$, represented by $S^2 \times \{\ast\} \times \{0\}$.

We can choose a $2$-dimensional sphere ${S^2}_0$ and smoothly embedded in the interior ${\rm Int}\ (S^2 \times S^2 \times D^1) \subset S^2 \times S^2 \times D^1$ in such a way that $2e$ is represented by it.
We can take a small closed tubular neighborhood $D_2$ of the sphere ${S^2}_0$.

By applying theory of so-called {\it handles} and {\it surgeries} we have a new $5$-dimensional closed and simply-connected manifold $D_{1,2}$ whose boundary consists of exactly two connected components. Furthermore, we can do so that the following properties are enjoyed.
\begin{itemize}
	\item The group $H_2(D_{1,2};\mathbb{Z})$ is isomorphic to the direct sum of $\mathbb{Z}$, $\mathbb{Z}$, and the commutative group of order $2$ or $\mathbb{Z}/2\mathbb{Z}$, $H_3(D_{1,2};\mathbb{Z})$ is the trivial group and $H_4(D_{1,2};\mathbb{Z})$ is isomorphic to $\mathbb{Z}$.
	\item The tangent bundle of $D_{1,2}$ is trivial and as a result $D_{1,2}$ can be smoothly immersed into ${\mathbb{R}}^5$.
	\end{itemize} 

More precisely, by removing the interior of $D_2$ from $D_1$ and attaching a copy of the manifold $D^3 \times S^2$ along the boundaries in a suitable way instead to have the manifold $D_{1,2}$. 

We can calculate the homology groups by using Mayer-Vietoris exact sequences for pairs of two manifolds 
$(D_1-{\rm Int}\ D_2,D_2)$ and $(D_1-{\rm Int}\ D_2,D_{1,2}-{\rm Int}\ (D_1-{\rm Int} D_2))$
 for example. To obtain $D_1$ and $D_{1,2}$ we glue the boundaries by (suitable) diffeomorphisms for these two pairs.
For some expositions on {\it handles} and {\it surgeries} see \cite{milnor3} for example.

 We apply Proposition \ref{prop:2} (\ref{prop:2}) to have a special generic map $f_0:M_0 \rightarrow {\mathbb{R}}^5$ on a suitable $6$-dimensional closed and simply-connected manifold $M_0$ such that $W_{f_0}$ in Proposition \ref{prop:2} (\ref{prop:2.2}) is diffeomorphic to $D_{1,2}$. We apply Proposition \ref{prop:3} to see that $M_0$ is simply-connected, which we know easily of course.

Proposition \ref{prop:8} completely determines the group $H_j(M_0;\mathbb{Z})$ for an arbitrary integer $j$. 
The rank of ${\rm Fr}(H_2(M_0;\mathbb{Z}))$ is $3$ and $H_2(M_0;\mathbb{Z})$ is isomorphic to the direct sum of $\mathbb{Z}$, $\mathbb{Z}$, $\mathbb{Z}$ and $\mathbb{Z}/2\mathbb{Z}$. 

 By Poincar\'e duality theorem for $M_0$, $H^4(M_0;\mathbb{Z})$ is isomorphic to $H_2(M_0;\mathbb{Z})$ and is also isomorphic to the direct sum of $\mathbb{Z}$, $\mathbb{Z}$, $\mathbb{Z}$ and $\mathbb{Z}/2\mathbb{Z}$.

We can also see that $H_2(M_0;\mathbb{Z}/2\mathbb{Z})$ and $H^2(M_0;\mathbb{Z}/2\mathbb{Z})$ are free modules over $\mathbb{Z}/2\mathbb{Z}$ of rank $4$ and so is $H^4(M_0;\mathbb{Z}/2\mathbb{Z})$.
We can consider the basis of $H_2(M_0;\mathbb{Z}/2\mathbb{Z})$, consisting of
exactly four elements in the following.
\begin{itemize}
	\item An element $e_1$ represented by the image of a section of the trivial smooth bundle over $S^2 \times \{{\ast}\} \times \{0\} \subset S^2 \times S^2 \times D^1 \subset D_{1,2}$ given by the restriction of $f$ to the preimage.
		\item An element $e_2$ represented by the image of a section of the trivial smooth bundle over $\{{\ast}\} \times S^2 \times \{\frac{1}{2}\} \subset S^2 \times S^2 \times D^1 \subset D_{1,2}$ given by the restriction of $f$ to the preimage.
	\item An element $e_3$ represented by the boundary of a fiber of the trivial bundle $D_2$ over ${S^2}_0$. This fiber is in the boundary of $D_1-{\rm Int}\ D_2 \subset D_{1,2}$. The boundary of $D_1-{\rm Int}\ D_2$ is in the inteiror of $D_{1,2}$.
	\item An element $e_4$ represented by ${q_f}^{-1}(L_0)$.
\end{itemize}	
	 We have the cohomology dual ${e_i}^{\ast}$ to $e_i$ for $i=1,2,3,4$ for the basis. Note that the first element $e_1$ can not be obtained from any element whose order is infinite where the coefficient ring is $\mathbb{Z}$ by changing the coefficient ring using the canonical quotient map from $\mathbb{Z}$ onto $\mathbb{Z}/2\mathbb{Z}$.
	 The remaining three elements are obtained in this way. By respecting Poincar\'e duality for $M_0$ and $D_{1,2}$, we have that the cup product ${e_2}^{\ast} \cup {e_4}^{\ast}$ is regarded as the Poincar\'e dual to $e_1$ under suitable orientations of the manifolds.
	 For Poincar\'e duals, we have the following facts.
	 \begin{itemize}
	 	\item For $i=1,2$, the Poincar\'e dual to ${e_i}^{\ast}$ is represented by the preimage $f^{-1}(T_i)$ of
	 	a suitable $3$-dimensional compact, connected and orientable manifold $T_i$ diffeomorphic to $S^2 \times D^1$ and embedded smoothly in $D_{1,2}$. Here the boundary of $T_i$ is regarded to be embedded into the boundary, two distinct connected components of the boundaries are regareded to be embedded in distinct connected components of the boundary, and the inteiror of $T_i$ is regarded to be embedded into the inteiror. We can regard that the preimage $f^{-1}(T_i)$ is the domain of a special generic map into ${\mathbb{R}}^3$ and a $4$-dimensional closed, connected and orientable manifold. 
	 	\item The Poincar\'e dual to ${e_3}^{\ast}$ is represented by the preimage $f^{-1}(T_3)$ of
	 	a suitable $3$-dimensional compact, connected and orientable submanifold $T_3$ diffeomorphic to the $3$-dimensional unit disk $D^3$ and embedded smoothly in $D_{1,2}$. Here the boundary of $T_3$ is embedded into the boundary and the inteiror of $T_3$ is embedded into the inteiror. Furthermore, $T_3 \bigcap (D_1-{\rm Int}\ D_2) \subset D_{1,2}$ is regarded as a manifold diffeomorphic to $S^2 \times D^1$. We can regard that the preimage $f^{-1}(T_3)$ is the domain of a special generic map into ${\mathbb{R}}^3$ and a $4$-dimensional closed, connected and orientable manifold and a standard sphere by Theorem \ref{thm:1} (\ref{thm:1.3},\ref{thm:1.4}). 
	 	\item The Poincar\'e dual to ${e_4}^{\ast}$ is represented by the image of a section over $S^2 \times S^2 \times \{0\} \subset S^2 \times S^2 \times D^1$ given by the restriction of $f$ to the preimage.  
	 
	 \end{itemize}	 

We have shown that the condition $m \geq 7$ is optimal. This completes the proof of Main Theorem \ref{mthm:1}.

We can show Main Theorem \ref{mthm:2} in a similar way to the proof of Main Theorem \ref{mthm:1} due to Proposition \ref{prop:6}.
\end{proof}
\begin{Ex}
For example, the product of $S^2$ and a $5$-dimensional closed and simply-connected manifold $M^{\prime}$ such that $H_2(M^{\prime};\mathbb{Z})$ is not free admits no special generic maps into ${\mathbb{R}}^n$ for $n=1,2,3,4,5$ by Main Theorem \ref{mthm:1} with Theorem \ref{thm:2}. In addition, for example, a manifold represented as connected sum of this manifold and any $7$-dimensional closed and simply-connected manifold admits no special generic maps into ${\mathbb{R}}^n$ for $n=1,2,3,4,5$ by a similar reason.
	
In Remark 3 of \cite{kitazawa2}, this is presented as an open problem.
\end{Ex}
\begin{Ex}
We can have various $6$-dimensional closed and simply-connected manifolds for Main Theorem \ref{mthm:2} by applying classification theory in \cite{jupp,wall,zhubr,zhubr2}.
\end{Ex}

\begin{proof}[A proof of Main Theorems \ref{mthm:3} and \ref{mthm:4}]
	Assume the existence of a special generic map $f:M \rightarrow {\mathbb{R}}^n$. 
	By Proposition \ref{prop:3}, $W_f$ is simply-connected.
	The proof of Main Theorems \ref{mthm:1} shows that for the map, we need to have a family $\{L_j\}$ of $1$-dimensional compact and connected manifolds diffeomorphic to a closed interval consisting of at least $l$ elements in Proposition \ref{prop:7}.
	 This completes the proof of Main Theorem \ref{mthm:3}.
   
   Prepare a special generic map in Example \ref{ex:1} (\ref{ex:1.2}) on $S^2 \times S^{m-4}$ into ${\mathbb{R}}^3$ and the product map of this and the identity map on $S^2$. We can smoothly embed the manifold of the target into ${\mathbb{R}}^5$ to have a special generic map from
    $S^2 \times S^2 \times S^{m-4}$ into ${\mathbb{R}}^5$. The singular set of the resulting map consists of exactly two connected components. 
    
    The number of connected components of the singular sets of special generic maps into ${\mathbb{R}}^5$ there must be at least $2$ by Theorem \ref{thm:5} or Main Theorem \ref{mthm:3}.

    Prepare $l_0>1$ copies of this map and applying a natural construction for a manifold represented as a connected sum of manifolds admitting special generic maps into a fixed Euclidean space, we have a special generic map on a manifold represented as a connected sum of $l_0$ copies of this manifold into ${\mathbb{R}}^5$ whose singular set consists of exactly $l_0+1$ connected components. For the construction, see \cite{saeki} for example. The numbers of connected components of the singular sets of special generic maps on the manifold  into ${\mathbb{R}}^5$ must be at least $l_0+1$ by virtue of Main Theorem \ref{mthm:3}.
    
    This complete the proof of Main Theorem \ref{mthm:4}. 
\end{proof}

\section{Acknowledgement.}
The author is a member of the project JSPS KAKENHI Grant Number JP17H06128 "Innovative research of geometric topology and singularities of differentiable mappings"
(Principal investigator: Osamu Saeki). The author is supported by this to produce the present study. 

We declare that we have no associated data essentially.

\end{document}